\documentclass[12pt,reqno]{amsart}

\numberwithin{equation}{section}

\usepackage{amsbsy}

\usepackage[colorlinks]{hyperref}
\hypersetup{
linkcolor=blue,          % color of internal links (change box color with linkbordercolor)
citecolor=green,        % color of links to bibliography
}
\usepackage[nobysame,abbrev,alphabetic]{amsrefs}

\usepackage{enumerate}
\usepackage{amssymb}
\usepackage{amsmath}
\usepackage{amscd}
\usepackage{amsthm}
\usepackage{amsfonts}
\usepackage{graphicx}
\usepackage[all]{xy}
\usepackage{verbatim}
\usepackage{hyperref}
\usepackage{gensymb}
\usepackage{mathrsfs}

\newtheorem{theorem}{Theorem}[section]

\newtheorem{lemma}[theorem]{Lemma}
\newtheorem{corollary}[theorem]{Corollary}
\newtheorem{cor}[theorem]{Corollary}
\theoremstyle{definition}\newtheorem{definition}[theorem]{Definition}

\newtheorem{proposition}[theorem]{Proposition}

\theoremstyle{definition}

\theoremstyle{definition}
\theoremstyle{definition}\newtheorem{remark}[theorem]{Remark}
\theoremstyle{definition}

\newcommand{\dd}{\mathbf{d}}
\newcommand{\al}{\alpha}

\newcommand{\ga}{\gamma}
\newcommand{\Ga}{\Gamma}
\newcommand{\del}{\delta}
\newcommand{\Del}{\Delta}

\newcommand{\Lam}{\Lambda}

\newcommand{\Om}{\Omega}
\newcommand{\vphi}{\varphi}
\newcommand{\vre}{\varepsilon}
    \newcommand{\Xn}{{\mathcal{L}_n}}
\newcommand\crly[1]{\mathscr{#1}}
\newcommand{\sm}{\smallsetminus}

\newcommand{\df}{{\, \stackrel{\mathrm{def}}{=}\, }}

\newcommand\Name[1]{\label{#1}{\ifdraft{\sn
      [#1]}\else\ignorespaces\fi}}

\newcommand\eq[2]{{\ifdraft{\ \tt
      [#1]}\else\ignorespaces\fi}\begin{equation}\label{#1}{#2}\end{equation}} 
\newcommand {\equ}[1]{\eqref{#1}}

\newcommand\Vol{\mathrm{Vol}}

\newcommand\covrad{\mathrm{covrad}}
\newcommand\conv{\mathrm{conv}}
\newcommand\supp{\mathrm{supp}}
\newcommand{\cA}{\mathcal{A}}
\newcommand{\cB}{\mathcal{B}}

\newcommand{\cD}{\mathcal{D}}
\newcommand{\cE}{\mathcal{E}}

\newcommand{\cI}{\mathcal{I}}

\newcommand{\cL}{\mathcal{L}}
\newcommand{\cM}{\mathcal{M}}

\newcommand{\cS}{\mathcal{S}}

\newcommand{\cU}{\mathcal{U}}
\newcommand{\cV}{\mathcal{V}}
\newcommand{\cW}{\mathcal{W}}

\newcommand{\cY}{\mathcal{Y}}

\newcommand{\spa}{{\rm span}}

\newcommand{\bR}{\mathbb{R}}
\newcommand{\bZ}{\mathbb{Z}}

\newcommand{\R}{{\mathbb{R}}}
\newcommand{\TT}{{\mathbb{T}}}

\newcommand{\Z}{{\mathbb{Z}}}

\newcommand {\ignore}[1]  {}

\newcommand{\SL}{\operatorname{SL}}

\newcommand{\defi}{\overset{\on{def}}{=}}
\newcommand\norm[1]{\left\|#1\right\|}

\newcommand\set[1]{\left\{#1\right\}}
\newcommand\pa[1]{\left(#1\right)}

\newcommand{\E}{\mathbf{e}}

\newcommand\av[1]{\left|#1\right|}
\newcommand\on[1]{\operatorname{#1}}

\newcommand\tb[1]{\textbf{#1}}
\newcommand\mat[1]{\pa{\begin{matrix}#1\end{matrix}}}
\newcommand\br[1]{\left[#1\right]}
\newcommand\smallmat[1]{\pa{\begin{smallmatrix}#1\end{smallmatrix}}}
\newcommand{\lra}{\longrightarrow}

\newcommand{\onto}{\xymatrix{\ar@{>>}[r]&}}
\newcommand{\da}[4]{\xymatrix{#1 \ar@<.5ex>[r]^{#2} \ar@<-.5ex>[r]_{#3} & #4}}
\newif\ifdraft\drafttrue
\draftfalse

\font\sn = cmssi8 scaled \magstep0

\newcommand{\nerve}{{\rm Nerve}}
\newcommand{\order }{{\mathrm {ord}}}
\newcommand{\Lb}{{\mathrm {Leb}}}
\newcommand{\mesh}{{\mathrm {mesh}}}

\newcommand{\asdim}{{\rm asdim}}

\begin{document}
\title[Closed orbits and well-rounded lattices]{Closed orbits for the diagonal group and well-rounded lattices}
\author{Michael Levin}
\address{Dept. of Mathematics, Ben Gurion University, Be'er Sheva,
  Israel
{\tt mlevine@math.bgu.ac.il}}
\author{Uri Shapira}
\address{Dept. of Mathematics, Technion, Haifa, Israel
{\tt ushapira@tx.technion.ac.il}
}
\author{Barak Weiss}
\address{Dept. of Mathematics, Tel Aviv University, Tel Aviv, Israel
{\tt barakw@post.tau.ac.il}}

\maketitle
\begin{abstract}
Curt McMullen showed that every compact orbit for the action of the
diagonal group on the space of lattices 
contains a well-rounded lattice. We extend this to all closed orbits. 
\end{abstract}
\section{Introduction}
Let $n \geq 2$ be an integer, let $G \df \SL_n(\R)%, \, \Gamma \df
%\SL_n(\Z)
$, let $A \subset G$ be the subgroup of diagonal
matrices with positive entries and let  $\Xn \df G/\SL_n(\Z)$ be the space of unimodular 
lattices in $\R^n$. The dynamics of the $A$-action on $\Xn$ is a
well-studied topic in view of applications to number theory. For
instance, McMullen \cite{McMullenMinkowski} studied this action in connection with his
fundamental work on Minkowski's conjecture.
A lattice $x \in \Xn$ is called {\em well-rounded} if the nonzero
vectors of shortest length in $x$ span $\R^n$, and McMullen proved that any compact $A$-orbit
contains a well-rounded lattice. We show: 

\begin{theorem}\Name{thm: main} 
Suppose $x \in \Xn$ and $Ax$ is closed. Then $Ax$ contains a
well-rounded lattice. 
\end{theorem}

The proof of Theorem \ref{thm:
  main} closely follows McMullen's strategy; namely McMullen deduced
theorem \ref{thm: main} from a covering result regarding covers of the
torus $\TT^n$, while we deduce it from a different covering
result. 

McMullen also showed that any {\em bounded} $A$-orbit contains a
well-rounded lattice {\em in its closure. } It is natural to inquire
whether this result could be strengthened, by removing either of the
italicized phrases. As we show in Proposition \ref{cor: almost every},
this strengthening holds for almost every $A$-orbit, but 
whether or not it holds for {\em every} $A$-orbit is an open
question. Our result could be seen as a partial step in this 
direction. Our proofs rely on results
of Tomanov and the authors \cite{TW, gruber} classifying closed
orbits for the $A$-action, as well as a
covering result which is of independent interest. For
another perspective on this and related questions, see \cite{PS}. 

\subsection{Acknowledgements} A preliminary version
of this paper appeared on arXiv as a chapter of the manuscript
\cite{gruber_arxiv}. We thank Roman
Karasev for enlightening remarks. The authors' work was
supported by ERC starter grant DLGAPS 279893 and ISF grants 190/08, 357/13, and 
the Chaya Fellowship.

\section{Preliminaries}\Name{sec: preparations}

\subsection{The behavior of almost every lattice}
Let $\mathcal{WR} \subset \Xn$ denote the set of 
well-rounded lattices. Suppose $x \in \mathcal{WR}$ and $v_1, \ldots,
v_n$ are  linearly independent shortest vectors of $x$. We will say that $x$ is a {\em
  generic well-rounded lattice} if for any $v \in x \sm \{0, \pm v_1,
\ldots, \pm v_n\}$, $\|v\| > \|v_i\|.$ For instance the lattice $\Z^n$
is generic well-rounded. We have:
\begin{proposition}\Name{prop: WR mfd}
If $x \in \Xn$ is a generic well-rounded lattice, then there is an
open $\cU \subset \Xn$ such that $\mathcal{WR} \cap \cU$ is a
submanifold of $\Xn$ of codimension $n-1$. 
\end{proposition}
\begin{proof}
The fact that $x$ is generic implies that there is a neighborhood
$\cV$ of the identity in $G$, such that for $g 
\in \cV$, any shortest vector of  $gx$ is $gv_i$ for some
$i$. Making $\cV$ smaller if necessary we obtain that the
multiplication map $g \mapsto gx$ is a homeomorphism of $\cV$ onto
$\cU = \cV x$, and we obtain that  $\mathcal{WR} \cap \cU = \{gx: \|gv_1\| = \cdots =
\|gv_n\|\}.$ Since $\{g \in G: \|gv_1\| = \cdots = \|gv_n\|\}$ is a
subvariety of $G$ cut out by $n-1$ independent equations,
$\mathcal{WR} \cap \cU$ is a submanifold of codimension $n-1$.
\end{proof}

\begin{proposition}\Name{cor: almost every}
For almost every $x \in \Xn$, the orbit $Ax$ contains a well-rounded
lattice. 
\end{proposition}
 
\begin{proof}
Let $x_0$ be a generic well-rounded lattice, with shortest vectors
$v_1, \ldots, v_n$, and let $\cU$ be the neighborhood of $x_0$ as in
Proposition \ref{prop: WR mfd}. Suppose in addition that derivatives of the maps $a \mapsto
\|av_i\|, \ i=1, \ldots , n-1$ are linearly independent (seen as linear functionals on the
Lie algebra $\mathfrak{a}$). A simple computation shows that this
condition is satisfied for the lattice $x_0 = \Z^n$. This condition
implies that the orbit $Ax_0$ and the manifold $\mathcal{WR} \cap
\cU$ intersect transversally at $x_0$. In particular there is a neighborhood
$\cU_0 \subset \cU$ such that if $x \in \cU_0$ then there is $a \in A$
such that $ax \in \mathcal{WR}$. Thus any $A$-orbit entering $\cU_0$
contains a well-rounded lattice. By Moore's ergodicity theorem (see
e.g. \cite{Zimmer}) the $A$-action on $\Xn$ is
ergodic, and hence almost every $A$-orbit enters $\cU_0$. 
\end{proof}

\subsection{The structure of closed orbits}
The following
statement was proved by the authors in \cite[Cor. 5.8]{gruber}, using earlier results of \cite{TW}. 

\begin{proposition}\Name{prop: TW, SW}
Suppose $Ax$ is closed. Then there is a decomposition $A = T_1
\times T_2$ and a direct sum decomposition $\R^n = \bigoplus_1^{d} V_i$
such that the following hold:
\begin{itemize}
\item
Each $V_i$ is spanned by some of the standard basis vectors. 
\item
$T_1$ is the group of linear transformations 
which act on each $V_i$ by a homothety, preserving Lebesgue measure on
$\R^n$. In particular $ \dim T_1 = d-1$. 
\item
$T_2$  is the group of diagonal
matrices whose restriction to each $V_i$ has determinant 1.  
\item $T_2 x$ is compact and $T_1 x$ is divergent;
  i.e. $Ax \cong T_1 \times T_2/(T_2)_x$, where
  $(T_2)_x \df \{a \in T_2: ax= x\}$ is cocompact in $T_2$. 
\item 
Setting $\Lambda_i \df V_i \cap x$, each $\Lambda_i$ is a
lattice in $V_i$, so that $\bigoplus \Lambda_i$ is of finite index in
$x$. 
\end{itemize}

\end{proposition}

\subsection{Some preparations}
Let $\E_1, \ldots, \E_n$ denote the standard basis of $\R^n$. For $1\le d\le n$,
let $$\tb{I}^n_d\defi\set{1\le i_1<\dots<i_d\le n}$$  
denote the collection of multi-indices of length $d$ and 
for $J = (i_1, \ldots, i_d)\in\tb{I}^n_d$ let 
$\E_J \df \E_{i_1} \wedge \cdots \wedge \E_{i_d}
$ be the corresponding vector in the $d$-th exterior power of $\R^n$. 
We equip $\bigwedge^d\bR^n$ with the inner product with
respect to which $\{\E_J\}$ is an orthonormal basis, and denote by $\cE_{d,n}$
the quotient of $\bigwedge^d\bR^n$ by the equivalence relation $w 
\sim -w$. Note that the product of an element of $\cE_{d,n}$ with a
positive scalar is well-defined. We will (somewhat imprecisely) refer to elements of $\cE_{d,n}$
as vectors. Given a subspace $L \subset \bR^n$  with $\dim L= d$,
we denote by 
  $w_L\in \cE_{d,n}$ the image of a vector of norm one in
 $ \bigwedge^d L.$
%\item 
If $\Lam \subset \bR^n$ is a discrete subgroup of rank $d$, we
  denote by $w_\Lam\in \cE_{d,n}$ 
the image of the vector 
$v_1\wedge\dots\wedge v_d,$ where $\set{v_i}_1^d$ forms a basis for
$\Lam$. The reader may verify that these vectors are well-defined
(i.e. independent of the choice of the $v_j$) and
satisfy $w_{\Lam} = |\Lam| w_L$ where $L = \spa \, \Lam$ and $|\Lam|$ is
the covolume of $\Lambda$ in $L$, with respect to the volume form
induced by the standard inner product on $\R^n$.  
%\item 
%\item 
We denote the natural action of $ G$ on $\cE_{d,n}$ arising from the
$d$-th exterior power of the linear action on $\bR^n$, by $(g, w)
\mapsto gw$. 
%\item 
Given a subspace $L \subset \bR^n$ and a discrete subgroup $\Lam$ we set 
$$A_L\defi\set{a\in A:
    aw_L=w_L} \text{ and } A_\Lam \df \{a \in A: aw_\Lam = w_\Lam\}.$$
Note that $A_L = A_{\Lam}$ when $\Lam$ spans $L$. 
Note also that the requirement  
$aw_L=w_L$ is equivalent to saying that $aL=L$ and $\det(a|_L)=1$.
%\item 
Given a flag 
\begin{equation}\label{flag}
\crly{F}=\set{ 0 \varsubsetneq L_1\varsubsetneq\dots\varsubsetneq L_k\varsubsetneq \bR^n}
\end{equation} 
(not necessarily full), let 
$A_{\crly{F}}\defi \bigcap_i A_{L_i}.$
%\end{enumerate}
%\end{definition}
%\begin{definition}
The {\em support} of an element $w \in \cE_{d,n}$ is the subset of
$\tb{I}^n_d$ for which the corresponding coefficients of an element of
$\bigwedge^d\bR^n$ representing $w$ are nonzero, and we write 
$\on{supp}(L)$ or $\on{supp}(\Lam)$ for the supports of $w_L$ and
$w_{\Lam}$. For $J = \left(i_1<\dots<i_d \right) \in \tb{I}^n_d$, set $\bR^J
\df \spa \, (\E_{i_j})$ and define the multiplicative characters 
$$ \chi_J: A \to \R^*, \ \chi_J(a) \defi \det
(a|_{\bR^J}).$$
Then 
% we see
%that the maps $\chi_J$ are homomorphisms to the multiplicative group
%$\R^*$, and since these homomorphisms are nontrivial and the $A$-action
%on each $\bigwedge^d \bR^n$ is $\bR$-diagonalizable, one easily checks 
%\end{definition}
%\begin{remark}\label{1636}
%As for any $a\in A$, 
%the standard basis $\set{e_J:J\in\tb{I}^n_d}$ consists
%of eigenvectors for $a$  (with eigenvalues $a_J\defi (\det
%a|_{\bR^J})$), we have that 
for any subspace $L\subset \bR^n$, 
\eq{1636}{A_L=\bigcap_{J \in \supp (L)} \ker \chi_J}
(and similarly for discrete subgroups $\Lam$). 
%On trying to verify the validity of hypothesis~\eqref{09302} of
%Theorem~\ref{thm: covering} the relevant  ``affine subspaces" will be
%of the form 
%$b\pa{\cap_i A_{L_i}}$ for some collections of subspaces $\set{L_i}\subset \bR^n$ and $b\in A$. 
%We now state and prove two Lemmas
%that will help us establish the \textit{almost affineness} in that hypothesis. The first will 
%give us the tool to bound the distance of certain elements from these
%affine spaces and the second will allow us to bound the 
%dimension of these affine spaces.
%We have the following two Lemmas that together are analogue to Theorem~\ref{finite distance from a group}. 
%As in \S \ref{establishing
%  topological input} w

%In order to
%verify hypothesis \eqref{09302} of Theorem \ref{thm: covering}, w
We fix an invariant metric $\dd$ on $A$. 
We will need the following
lemma (cf. \cite[Theorem 6.1]{McMullenMinkowski}): 
\begin{lemma}\label{bdd dist from stab}
Let $T \subset A$ be a closed subgroup and let $x\in\cL_n$ be a
lattice with a
compact $T$-orbit. Then for any $C>0$ there exists  
$R>0$ 
such that for any collection $\set{\Lam_i}$ of subgroups of $x$, there exists $b\in A$ such that 
\begin{equation}\label{2052}
\left \{a\in T:\forall i, \; \norm{aw_{\Lam_i}}\le C \right\} \subset 
\set{a\in A: \dd\left(a,\, b  \left(\cap_iA_{\Lam_i} \right) \right)\le R}.
%R
%\text{-neighborhood of } b\pa{\bigcap_iA_{\Lam_i}}.
\end{equation}
\end{lemma}
\begin{proof}
In the argument below we will sometimes identify $A$ with its Lie algebra $\mathfrak{a}$ via the
exponential map, and think of the subgroups $A_\Lam$ as
subspaces. %We will also assume 
%with no loss of generality that the distance function $d$ %denote an invariant distance function on $A$ (more concretely, we take
%$d$ to be
%is %induced by 
%the euclidean norm on $\mathfrak{a}$. 
By \equ{1636} only finitely many subspaces arise as $A_\Lam$.
%\begin{equation}\label{1637}
%\bigcap_i A_{L_{\Lam_i}}= \bigcap_{J \in \cup \supp (\Lam_i)} \ker \chi_J.
%\end{equation}
In particular, given a collection of discrete subgroups
$\set{\Lam_i}$, the angles between the spaces they span (if nonzero) are bounded
below, by a bound which is independent of the $\set{\Lam_i}$. Therefore
%, for a sequence
%$\set{a_n}$ in $A$,
%\begin{align*}
%d(a_n, \cap_i A_{\Lam_i})\to\infty &\Longleftrightarrow \max_i
%d(a_n, A_{\Lam_i})\to\infty\\ 
%&\overset{\eqref{1636}}{\Longleftrightarrow
%}\exists\; J\in\cup_i \on{supp}(w_{L_{\Lam_i}}),\; \chi_J(a_n)\to
%0\textrm{ or }\infty. 
%\end{align*}
%From here it follows that
there exists a function $\psi:\bR\to\bR$ with $\psi(R)\to_{R\to\infty}\infty$, such that   
\begin{align}\label{336}
&\set{a\in A: \forall\; J\in\cup_i \on{supp}(w_{\Lam_i}),\; \psi(R)^{-1}\le \chi_J(a)\le \psi(R)}\subset\\
\nonumber &\set{a\in A: \dd(a,\cap_iA_{\Lam_i})\le R}.
\end{align}

Since $Tx$ is compact,  there exists a compact subset
$\Om\subset T$ %satisfying $\Om x= Tx$, i.e. 
such that for any $a\in T$ there exists $b = b(a) \in T$ satisfying $bx=x$ and $b^{-1}a\in\Om$.
It follows that there exists $M \geq 1$ such that:
\begin{enumerate}[(I)]
\item\label{1708} for any subspace $L$, $||bw_L||\le M||aw_L||$. 
\item\label{1747} for any multi-index $J$, $\chi_J(ba^{-1})\le M$.
\end{enumerate}

Given $C>0$, let $C' \df MC$ and consider the finite set 
$$\crly{S}\defi \set{\Lam \subset x: ||w_{\Lam}||\le C'}.$$
For any $\Lam\in\crly{S}$ write 
$w_{\Lam}=\sum_{J\in\on{supp}(w_{\Lam})} \al_J(\Lam) \E_J.$ 
Let 
\begin{align*}
0< \vre&<\min\set{\av{\al_J(\Lam)}:\Lam\in\crly{S}, J\in\on{supp}(w_{\Lam})},
%\vre^{-1}&>\max\set{\av{\al_J(\Lam)}:\Lam\in\crly{S}, J\in\on{supp}(w_{\Lam})}.
\end{align*}
and choose $R$ large enough so that $\psi(R)>C'/\vre$. 
We claim that 
for any $\set{\Lam_i}\subset \crly{S}$,
\begin{equation}\label{2053}
\set{a\in T:\forall i\; \norm{aw_{\Lam_i}}\le C}\subset \set{a\in T: \dd(a, \cap_i A_{\Lam_i})\le R}.
\end{equation}
To prove this claim, suppose $a$ is an element on the left hand side
of~\eqref{2053}. 
By~\eqref{336}
% in order to establish its inclusion in the set on the 
%right hand side of~\eqref{2053} 
it is enough to show that for any $J\in\cup_i\on{supp}(\Lam_i)$ we have
$\psi(R)^{-1}\le \chi_J(a)\le\psi(R)$.  
% To this end fix $i$ and
%$J\in\on{supp}(\Lam_i)$. 
Since the coefficient of $\E_J$ in the
expansion of $aw_{\Lam_i}$ is $\chi_J(a)\al_J(\Lam_i)$ and since 
%then the fact that 
$||aw_{\Lam_i}||\le C$, we have 
$$\chi_J(a)\le
\frac{C}{|\al_J(\Lam_i)|}\le\frac{C}{\vre}\le \psi(R).$$ 
On the other hand, letting $b = b(a)$ 
%by Step 2 we know that there exists $b\in T$ such
%that $bx=x$ and  
%statements~\eqref{1708},\eqref{1747}  
%hold. 
we have $b\Lam_i\in\crly{S}$ from \eqref{1708}, and 
\begin{align*}
 \vre\le |\al_J(b\Lam_i)| & =\chi_J(b) |\al_J(\Lam_i)| \ \Longrightarrow \ \chi_J(b^{-1})\le C/\vre \\
&\overset{\textrm{\eqref{1747}}}{\Longrightarrow} \ \chi_J(a^{-1})=\chi_J(a^{-1}b)\chi_J(b^{-1})\le C'/\vre\le\psi(R),
\end{align*}
which completes the proof of \eqref{2053}.

%Let $C>0$ be given and let $R>0$ be chosen as shown in Step 3. 
Let $\set{\Lam_i}$ be any collection of subgroups of $x$
and assume that the set on the left hand side of ~\eqref{2052} is non-empty. That is, there 
exists $a_0\in T$ such that for all $i$, $||a_0w_{\Lam_i}||\le C$. 
Let $b = b(a_0)\in T$, and set $\Lam'_i \df b\Lam_i$. It follows that $\set{\Lam'_i}\subset\crly{S}$ 
and so 
\begin{align*}
\set{a\in T:\forall i\norm{aw_{\Lam_i}}\le C} 
&=b\set{a\in T: \forall i \norm{aw_{\Lam'_i}}\le C}\\
&\stackrel{\eqref{2053}}{\subset} b \set{a\in T: \dd(a,\cap_i A_{\Lam_i'})\le R}\\
&= \set{a\in T: \dd(a, b\pa{\cap_i A_{\Lam_i}})\le R},
\end{align*}
where in the last equality we used the fact that
$A_{\Lam_i'}=A_{\Lam_i}$ because $A$ is commutative.
\end{proof}

\begin{lemma}\Name{lem: flag}
Let $\crly{F}$ be a flag of length $k$ as in \eqref{flag} and let
$A_{\crly{F}}$ be its stabilizer. Then $A_{\crly{F}}$ is of
co-dimension  
$\ge k$ in $A$.
\end{lemma}
\begin{proof}
Given a nested sequence of multi-indices
$J_1\varsubsetneq\dots\varsubsetneq J_k$ it is clear that the subgroup
$$
\bigcap_{i=1}^k \ker \chi_{J_i}
$$
is of co-dimension $k$ in $A$. In light of \eqref{1636},
it suffices to prove the following claim: 
\quad\\
\noindent \textit{Let $\crly{F}$ be a flag as
  in~\eqref{flag} with $d_i\defi \dim L_i$. Then there is a nested
  sequence 
of multi-indices $J_i\in\tb{I}^n_{d_i}$ such that $J_i\in\on{supp}(L_i)$.}
\quad\\

%\noindent Once the Claim is proved, we have that $A_{\crly{F}}$ is a
%subgroup of a group of co-dimension $k$ and therefore is  
%of co-dimension $\ge k$ as desired.

In proving the claim we will assume with no loss of generality that
the flag is complete. Let $v_1,\dots, v_n$ be a basis of $\bR^n$ such
that  $L_i=\on{span}\set{v_j}_{j=1}^i$ for 
$i=1,\dots, n-1.$  
Let $S$ be the $n\times n$ matrix whose columns are $v_1, \dots, v_n$.
Given a multi-index $J$ of length $\av{J}$, we denote by $S_J$
the square matrix of dimension $\av{J}$ obtained  
from $S$ by deleting the last $n-\av{J}$ columns and the rows
corresponding to the indices not in $J$.  Note that with this
notation, each $w_{L_d}$ is the image in $\cE_{d,n}$ of a vector proportional to 
\begin{equation}\label{1159}
%w_{L_d}=
v_1\wedge\dots\wedge v_d=\sum_{J\in \tb{I}^n_d}  (\det S_J) \E_J.
\end{equation}
In particular,  $J\in\on{supp}(L_d)$ if and only if $\det S_J\neq 0$.

Proceeding inductively in reverse, we construct the nested sequence
$J_d$ by induction on $d =n, \ldots, 1$. Let $J_n=\set{1,\dots, n}$ so that
$S=S_{J_n}$.  
Suppose we are given  multi-indices
$J_{n}\supset\dots\supset J_{d+1}$ such that  
$J_i\in\on{supp}(w_{L_i})$ 
for $i=n,\dots, d+1$. We want to define now a multi index
$J_d\in\on{supp}(w_{L_d})$ which is  
contained in $J_{d+1}$. By~\eqref{1159}, $\det S_{J_{d+1}}\neq 0$. When
computing $\det S_{J_{d+1}}$ by expanding the last column we express
$\det S_{j_{d+1}}$ 
as a linear combination of $\set{\det S_J:J\subset J_{d+1},
  \av{J}=d}$. We
conclude that there must exist at least one multi-index $J_d\subset
J_{d+1}$  for which $\det S_{J_d}\ne 0$. In turn, by~\eqref{1159} 
this means that $J_d\in\on{supp}(w_{L_d})$. This finishes the proof
of the claim. 
\end{proof}

\section{Reduction to a topological statement} \Name{sec:
  closed orbits} 
\ignore{

It is an immediate consequence of Theorem \ref{thm: main} that any
closed $A$-orbit contains a stable lattice. The purpose of this
section is to show that the same is true for the set of well-rounded
lattices. 
Note that this was proved by McMullen for {\em compact} orbits but
for general  closed orbits, does not follow from his results. Our
proof relies on previous work of Tomanov and the second-named author
\cite{TW}, on \cite{gruber}, and on a covering result (communicated to
the authors by Michael Levin), whose proof is given in the appendix to
this paper. 

\begin{theorem}\Name{thm: closed orbits}
For any $n$, any closed orbit $Ax \subset \Xn$ contains a well-rounded lattice. 
\end{theorem}
}
We will require the following topological result which generalizes 
Theorem 5.1 of~\cite{McMullenMinkowski}. 
Let $s,t$ be non-negative integers, and let $\Delta$ denote the 
$s$-dimensional simplex, which we think of concretely as $\mathrm{conv} (\E_1,
\ldots, \E_{s+1})$, where the $\E_j$ are the standard basis vectors in
$\R^{s+1}$. % Also let $\TT$ denote the $t$-dimensional torus. 
We will discuss covers of $M \df \Delta \times \R^t $, and give conditions
guaranteeing that such a cover must cover a point at least $s+t+1$
times. %Let $\cU$ be a cover of $\TT \times \Delta$ by connected
%open sets. 
%We will say that $\cU$ has positive inradius if Definition
%\ref{def: inradius} holds with $\R^n \times \Delta$ in place of $A$, and say that
%$U \subset \R^n \times \Delta$ is $(R,k)$-{\em almost affine} if it is
%contained in the $R$-neighborhood of a set of the form $V \times
%\Delta$ where $V$  is a $k$-dimensional affine subspace of $\R^n$. 
For
$j=1, \ldots, s+1$ let $F_j$ be the face 
of $\Delta$ opposite to $\E_j$, that is $F_j = \mathrm{conv} (\E_i: i
\neq j)$. Also let $M_j \df F_j \times  \R^t $ be the corresponding
subset of $M$. Given $R>0$ and a positive integer $k$, we say that a
subset $U\subset \R^t$ is $(R,k)$-\textit{almost affine} if it is 
contained in an $R$-neighborhood of a $k$-dimensional affine subspace of $\R^t$. 

\begin{theorem}%[Michael Levin]
\Name{thm: covering}
Suppose that  $\cU$ is a cover of $M$ 
by open sets 
%with positive inradius, \marg{Maybe inradius can be removed.} 
satisfying the following conditions:
\begin{enumerate}[(i)]
\item\Name{09301}
For any connected component $U$ of any element of  $\cU$ there exists $j$ 
such that $U \cap M_j = 
\varnothing.$  
\item\Name{09302}
There is $R$ so that for 
any connected component $U$ of the intersection of $k \leq s+t$ distinct 
%connected components of 
elements of $\cU$, 
the projection of $U$ to $\R^t$ is $(R, s+t-k)$-almost
affine.
\end{enumerate}
Then
there is a point of $M $ which is covered at least
$s+t+1$ times.

\end{theorem}

%%%%%%%%%%%%%%%%%%%%%%%%%%
%%%%%%%%%%%%%%%%%%%
%                                  This is the place to start the cut and paste from
%%%%%%%%%%%%%%%%%%%%%
%%%%%%%%%%%%%%%%%%%%%
The case $s=0$ is McMullen's result, and the case $t=0$ is known as
the Knaster-Kuratowski-Mazurkiewicz theorem (see e.g. \cite{karasev}). 
Note that hypothesis~\eqref{09302} is trivially satisfied when $k \leq s $,
since any subset of $\R^t$ is $(1, t)$-almost affine. 
We will prove Theorem \ref{thm: covering} in \S \ref{appendix: Levin}. In this
section we use it to
prove Theorem \ref{thm: main}.

%The following notation is analogous to Definition \ref{bn}.
Given a lattice $x\in \Xn$ let $\alpha(x)$ denote the
length of a shortest nonzero vector in $x$. Given $\del>0$ let 
\begin{align*}
\on{Min}_{\del}(x)&\defi\set{v \in x \sm \{0\} : \|v\|<(1+\del)\al(x)}\\
\tb{V}_{\del}(x)&\defi\on{span}\on{Min}_{\del}(x)\\
\dim_\del(x)&\defi\dim\tb{V}_{\del}(x).
\end{align*}
Finally, for $\vre>0$, let $\cU^{(\vre)}=\set{U_j^{(\vre)}}_{j=1}^n$
be the collection of open subsets of $A$ defined by  
\eq{eq: cover}{ U_j = U_j^{(\vre)} \df \{a \in A: \text{for all }
  \delta \text{ in a neighborhood of }
j\vre, \, \dim_{\del} (ax) =j \}. }
Note that these sets depend on $x \in \Xn$ but in our application $x$
will be considered fixed and so we suppress this dependence from our
notation. 
By \cite[Thm. 7.2]{McMullenMinkowski}, 
%Similarly to the discussion in Lemma \ref{lem: positive inradius} we see that 
$\cU^{(\vre)}$ is an open cover of $A$.
\begin{lemma}\Name{lem: compactness}
For any $n$ there is a compact $K \subset \Xn$ such that if $x \in
\Xn$ and $a \in U_n^{(\vre)}$ for $\vre <1$, then $ax \in K$.  
\end{lemma}

\begin{proof}
If $a \in U_{n}^{(\vre)}$ and $\vre <1$ then $ax$ has $n$ linearly
independent vectors of length at most $(n+1)\alpha(ax)$. Since $ax$ is
unimodular, there is a constant $C$ (depending only on $n$) such that
$\alpha(x) \geq C$. The set $K \df
\{x \in \Xn: \alpha(x) \geq C\}$ is compact by Mahler's compactness
criterion, and fulfills the requirements.
\end{proof}

\begin{proof}[Proof of Theorem \ref{thm: main}]  
It suffices to show that for each $\vre>0$, 
$U_n^{(\vre)} \neq
  \varnothing.$ Indeed, if this is the case, then 
letting $\vre_j$ be a sequence of positive numbers such that 
$\vre_j
  \to 0$, for each $j$, we let $a_j \in U_n^{(\vre_j)} $. 
By Lemma \ref{lem: compactness}, the lattices $a_jx$ belong to a fixed
compact set $K$, so there is subsequence converging to some $x_0 \in
K$; we continue to denote the subsequence by $(a_jx)$. For each $j$ there are
linearly independent $v^{(j)}_1, \ldots, v^{(j)}_n \in x$ such that
$\left\|av^{(j)}_i \right\| \leq (1+\vre_j)\alpha(a_jx)$. The angle between each
$a_j v^{(j)}_i$ and the space spanned by the other $a_j v^{(j)}_\ell,
\ell \neq i$ is bounded from below independently of $j$. Passing to a
subsequence we can assume that each $a_j v^{(j)}_i$ converges to a
nonzero vector $v_i \in x_0$. Since $\alpha$ is a continuous function,
the $v_i$ all have length equal to $\alpha(x_0)$, and by the lower
bound on the angles between them, they are linearly independent; that
is, $x_0$ is well-rounded. 

In order to prove that $U_n^{(\varepsilon)}$ is non-empty, we will
apply Theorem~\ref{thm: covering}. 
The first step %towards applying Theorem~\ref{thm: covering} 
is to find a decomposition
 $A\simeq\bR^{n-1}=\bR^s\times\bR^t$ 
and a simplex $\Del\subset \bR^s$, so that the restriction of the cover to
$\Del\times\bR^t$ satisfies the two hypotheses of Theorem~\ref{thm: covering}.

Let $A = T_1 \times  T_2$ and $\R^n = \bigoplus_1^{d} V_i$ be the
decompositions as in Proposition 
\ref{prop: TW, SW}, and let $s \df \dim T_1 = d-1$. 
\ignore{
, there is a
decomposition $A = T_1
\times T_2$ and
such that the following hold:
\begin{itemize}
\item
Each $V_i$ is spanned by some of the standard basis vectors. 
\item
$T_1$ is the group of linear transformations 
which act on each $V_i$ by a homothety, preserving Lebesgue measure on
$\R^n$. In particular $s \df \dim T_1 = d-1$. 
\item
$T_2$  is the group of diagonal (with respect to the standard basis)
matrices whose restriction to each $V_i$ has determinant 1.  
\item $T_2 x$ is compact and $T_1 x$ is divergent;
  i.e. $Ax \cong T_1 \times T_2/(T_2)_x$, where
  $(T_2)_x \df \{a \in T_2: ax= x\}$ is cocompact in $T_2$. 
\item 
Setting $\Lambda_i \df V_i \cap x$, each $\Lambda_i$ is a
lattice in $V_i$, so that $\bigoplus \Lambda_i$ is of finite index in
$x$. 
\end{itemize}
}
For $a \in T_1$ we denote by $\chi_i(a)$ the number
satisfying $av = e^{\chi_i(a)}v$ for all $v \in V_i$. Thus each $\chi_i$
is a homomorphism from $T_1$ to the additive group of real
numbers. The mapping $a \mapsto \bigoplus_i \chi_i(a)
\mathrm{Id}_{V_i}$, where $\mathrm{Id}_{V_i}$ is the identity map on
$V_i$,  is nothing but the logarithmic map of $T_1$ and it endows
$T_1$ with the structure of a vector space. In particular we can discuss the
convex hull of subsets of $T_1$. 
For each $\rho$ we let
$$\Delta_\rho \df \{a \in T_1: \max_i \chi_i(a) \leq \rho\}.$$  
Then
$\Delta_\rho = \conv (b_1, \ldots, b_d)$ where $b_i$ is the diagonal
matrix acting on each $V_j, \, j \neq i$ by multiplication by $e^\rho$, and
contracting $V_i$ by the appropriate constant ensuring that $\det b_i
=1$.

Let $P_i : \R^n \to 
V_i$ be the natural projection associated with the decomposition $\R^n =
\bigoplus V_i$. Since $\bigoplus \Lambda_j$ is of finite index in $x$, each
$P_i(x)$ contains $\Lambda_i$ as a subgroup of 
finite index and hence is discrete in $V_i$. Moreover, the orbit $T_2 x$ is compact, so
for each $a \in T_2$ there is $a'$ belonging to a bounded subset of $T_2$
such that $ax=a'x$. This implies that there is
$\eta>0 $ such that for any $i$ and any $a \in T_2$, if $v \in ax$ and $P_i(v) \neq 0$
then $\| P_i(v)\| \geq \eta$. Let $C>0$ be large enough so that
$\alpha(x') \leq C$ for any $x' \in \Xn$.  Let
$\rho$ be large enough so that 
\eq{eq: choice of R}{e^\rho\eta > 
2C. }
We restrict the covers $\mathcal{U}^{(\vre)}$
(where 
$\vre \in (0,1/n)$) to $\Delta_\rho \times
T_2$ and apply Theorem \ref{thm: covering} with $t \df \dim T_2 = n-d$.

 Let   $U$ be a connected  subset of 
  $U_k^{(\vre)} \in \cU^{(\vre)}$.
By  %Repeating the arguments proving Lemma   \ref{flat things}, or
    %appealing to 
\cite[\S7]{McMullenMinkowski},% we see that 
the $k$-dimensional subspace 
  $L\defi a^{-1}\tb{V}_{k\vre} (ax)$ 
  as well as the discrete subgroup $\Lam\defi L\cap x$
  are independent of the choice of $a \in U$. 
  By definition of $U_k^{(\vre)}$, for any $a \in U$, $a \Lam$ contains $k$ vectors $v_i = v_i(a),
  i=1, \ldots,
  k$ which span $aL$
  and satisfy \eq{eq: vi satisfy}{
\|v_i\| \in [r, (1+k\vre)r ], \ \ \text{where \ } r \df  \alpha(ax).
 } 
 
In order to 
verify hypothesis~\eqref{09301} of Theorem \ref{thm:
    covering}, we need to show that there 
is at least one $j$ for which $U \cap M_j  = 
\varnothing$. Since 
 $\ker P_1 \cap \cdots \cap \ker P_d = \{0\}$ and $\dim L = k \geq 1$, it suffices to show that
 whenever $U \cap M_j \neq \varnothing$, $L \subset \ker P_j$. 
The face $F_j$ of $\Delta_\rho$ consists of those elements $a_1 \in T_1$
which expand vectors in $V_j$ by a factor of 
$e^\rho$. If $U \cap M_j \neq \varnothing$ then there is $a \in T_2, a_1
\in F_j$ so that $a_1a \in U$. Now \equ{eq: choice of R}, \equ{eq: vi
  satisfy} and the choice of $\eta$ and $C$  ensure that 
the vectors $v_i = v_i(a_1a)$ satisfy $P_j(v_i)=0$. Therefore $L \subset
\ker P_j$.  

It remains to 
verify hypothesis~\eqref{09302} of Theorem \ref{thm: covering}. 
Let $U$ be a connected subset of an intersection $U_{i_1}\cap\dots\cap
U_{i_k}\cap(\Del_\rho\times T_2)$ and let  
$L_{i_j}\defi a^{-1}\tb{V}_{i_j\vre} (ax)$ and $\Lam_{i_j}\defi L_{i_j}\cap x$. 
As remarked above, $L_{i_j},\Lam_{i_j}$ are independent of $a\in U$.

By the definition of the $L_{i_j}$'s we have that $L_{i_j}\varsubsetneq L_{i_{j+1}}$ and so they form 
a flag $\crly{F}$ as in~\eqref{flag}. Lemma~\ref{lem: flag} applies and we deduce that 
\begin{equation}\label{1640}
A_{\crly{F}}=\cap_{j=1}^k A_{L_{i_j}} \textrm{ is of co-dimension}\ge k\textrm{ in }A.
\end{equation} 
For each $a\in U$ and each $j$ let $\set{v^{(j)}_\ell(a)}\in a\Lam_{i_j}$ be the vectors spanning 
$aL_{i_j}$ which satisfy~\eqref{eq: vi satisfy}. Let 
$u^{(j)}_\ell(a)\defi a^{-1} v^{(j)}_\ell\in\Lam_{i_j}$. Observe that:
\begin{enumerate}[(a)]
\item\label{02281} $\on{span}_{\bZ}\set{u^{(j)}_\ell(a)}$ is of finite
  index in $\Lam_{i_j}$ and in particular, 
$u^{(j)}_{i_1}(a)\wedge\dots\wedge u^{(j)}_{i_j}(a)$ is an integer
multiple of $\pm w_{\Lam_{i_j}}$. As a consequence
$\left\|aw_{\Lam_{i_j}} \right\|\le
\left \|v^{(j)}_{i_1}(a)\wedge\dots\wedge v^{(j)}_{i_j}(a) \right\|$. 
\item\label{02282} Because of~\eqref{eq: vi satisfy} we have that
  $\left \|v^{(j)}_{i_1}(a)\wedge\dots\wedge v^{(j)}_{i_j}(a)
  \right\|< C$ for some constant $C$ depending on $n$ alone.
\end{enumerate}
It follows from~\eqref{02281},\eqref{02282} and Lemma~\ref{bdd dist from stab}
that there exist $R>0$ and an element $b\in T_2$ so that
$$U\subset \Del_\rho \times \set{a\in T_2:\forall i_j,
  \left\|aw_{\Lam_{i_j}}\right\|<C}\subset T_1\times \set{a\in T_2:
  \dd(a,bA_{\crly{F}})\le R}.$$ 
%%%%%%%%%% 
By~\eqref{1640} we deduce that 
if $p_2 : A \to T_2$ is the projection 
associated with the
decomposition $A = T_1 \times T_2$ then $p_2(U)$ is
$(R',s+t-k)$-almost affine, where $R'$ depends only on $R,\rho$. This
concludes the proof.  
\end{proof}

%\appendix
\section{Proof of Theorem \ref{thm: covering}}
\Name{appendix: Levin}
In this section we will prove Theorem \ref{thm:  covering}. Our proof 
gives an elementary alternative proof of McMullen's result.
 Moreover
it shows that McMullen's hypothesis that the inradius of the cover is
positive, is not essential.

Below $X$ will denote a second countable locally connected metric space.
We will use calligraphic letters like $\mathcal{U}$ for collections of
sets. The symbol 
 $\mesh ( \mathcal {A})$ will 
 denote the supremum of the diameters of the sets in
 $\mathcal A$.
The symbol $\Lb (\mathcal{A})$ will denote
 the Lebesgue number of a cover $\mathcal A$, i.e. the supremum of all
 numbers $r$ such that each ball of radius $r$ in $X$ is contained in
 some element of $\cA$. The symbol
 $\order(\mathcal{A})$ will denote the  
 largest number of distinct elements of $\mathcal A$ with non-empty
 intersection. 
\begin{definition}\Name{def: asdim}
 Let $\{X_j\}_{j \in \crly{J}}$ be a collection of subsets of $X$. We
 consider each $X_j$ as an independent metric space with the metric
 inherited from $X$, and say that the collection is {\em uniformly of 
 asymptotic dimension $\leq n$}  if for every $r>0$ 
 there is $R >0$ such that for every $j \in \crly{J}$ there is 
  an open cover ${\mathcal X}_j $  of $X_j$ such that
\begin{itemize}
\item  $\mesh ( {\mathcal X}_j) \leq R$.
\item $\Lb ({\mathcal X }_j)> r$. 
\item$\order( {\mathcal X}_j)\leq n+1$. 
\end{itemize}

As an abbreviation we will sometimes write `asdim' in place of `asymptotic
dimension'. 
\end{definition}

Recall that a cover of $X$ is {\em locally finite} if every $x \in X$
has a neighborhood which intersects finitely many sets in the cover. 
We call the intersection of $k$ distinct elements of $\cA$ a {\em
  $k$-intersection,} and denote the union of all $k$-intersections by 
 $[{\mathcal A}]^{k}$. We will need the following two Propositions for
 the proof of Theorem~\ref{thm: covering}. We first prove  
 Theorem~\ref{thm: covering} assuming them and then turn to their proof.

\begin{proposition}
\Name{p2}
Let $\mathcal A$ be a locally finite open cover of  $X$
such that $\order ({\mathcal A}) \leq m$ and the collection of components of
the  $k$-intersections 
of 
$\mathcal A$, $1 \leq k \leq m$,  is uniformly  of $\asdim \leq m-k$.
Then $\mathcal A$ can be refined by a uniformly bounded 
open cover of order at most  $m$.
\end{proposition}

Note that McMullen's theorem, namely the case $s=0$ of Theorem
\ref{thm: covering}, already follows from Proposition \ref{p2}, since
by a theorem of Lebesgue, a uniformly bounded open cover of $\R^t$ is
of order at least $t +1$.  

 \begin{proposition}
 \Name{p3}
 Let $\Delta_1$ and $\Delta_2$ be simplices, $X=\Delta_1 \times \Delta_2$,
 $p_i : X \to \Delta_i$ the projections and 
  $\mathcal A$  a  finite open cover of $X$ such that
  for every $A \in \mathcal A$ and $i=1,2$ the set $p_i(A)$ does 
  not meet at least one of the faces of $\Delta_i$.
  Then $\order ({\mathcal A}) \geq \dim \Delta_1 + \dim \Delta_2 +1 $.
  \end{proposition}
 
%\subsection{Proofs}

 \begin{proof}[Proof of Theorem \ref{thm: covering}.] 
Let $m \df \dim M = s+t$, and suppose by contradiction that 
 $\order ({\mathcal U}) \leq  m$. Since every cover of $M$ has a
 locally finite refinement, there is no loss of generality in assuming
 that $\cU$ is locally finite. Replacing $\cU$
   with the set of connected components of elements of $\cU$, we may 
   assume that all elements of $\cU$ are connected. For any $r_0$, and
   any bounded set $Y$, the
   product space  $Y \times \R^d$ can be covered by a cover of order $d+1$ and
   Lebesgue number greater than $r_0$. Hence our hypothesis (ii) implies that for each
 $k =1, \ldots, m$, the collection of connected components of intersections of $k$
 distinct elements of $\cU$ is
 uniformly of
 asymptotic dimension at most $m-k$. 
Therefore we can apply Proposition
 \ref{p2} to assume that  $\mathcal U$ 
  is  uniformly bounded and of order at most $m$. 
  Take a sufficiently large $t$-dimensional simplex  $\Delta_1 \subset \R^t$ so that 
   the projection of every set  in $\mathcal U$  does not
  intersect at least one of the faces of $\Delta_1$. 
We obtain a contradiction to 
   Proposition \ref{p3}.
\end{proof}
For the proofs of Propositions \ref{p2}, \ref{p3} we will need some auxiliary
lemmas. 
\begin{lemma}\Name{lem: for p2}
Let $\{G_i: i \in \crly{I}\}$ be a locally finite collection of open subsets
of $X$,
and let $Z$ be an open subset such that for each $i \neq j$, $G_i
\cap G_j \subset Z$. Then there are disjoint open subsets
$\mathbf{E}_i, i \in \crly{I},$ such that for any $i$
$$G_i \sm Z \subset \mathbf{E}_i
\subset G_i.
$$
\end{lemma}

\begin{proof}
\ignore{
For each $G_i$ and 
$x \in G_i \cap \partial
\, (G_i \sm Z)$ 
set   
$$
r_x \df \frac{1}{3} \inf_{j \neq i} \dd\left(x, \partial
\, \left(G_j  \sm Z\right) \right),
$$
where $\dd$ is the metric on $X$. The infimum in this
definition is in fact a minimum since $\{G_i\}$ is locally
finite. To see that it
is positive, suppose if possible that  $y_\ell \to x$
for a
sequence $(y_\ell) \subset \partial ( G_j \sm Z)$. 
Then there are $\tilde{y}_\ell \in G_j \sm Z$ with $\dd(y_\ell, \tilde{y}_\ell) \to 0$ so that
$\tilde{y}_\ell \to x$.  Since $G_i$ is open, for large enough $\ell$ we have
$\tilde{y}_\ell \in G_i$, contradicting the assumption that $G_i \cap G_j
\subset Z$. 
Now we set 
$$
\mathbf{E}_i \df \mathbf{E}' \cup \mathbf{E}'', \ \ \text{where \ \ }
\mathbf{E}' \df G_i \sm Z \ \ \text{and \ } \mathbf{E}'' \df G_i \cap
\bigcup_{x \in  G_i \cap \partial
\, (G_i \sm Z)} B(x,r_x).
$$
Clearly each $\mathbf{E}_i$ satisfies \equ{eq: inclusion}, and it is
open since $\mathbf{E}''$ is open and covers the boundary points of
$\mathbf{E}'$. To show that the sets $\mathbf{E}_i$ are 
disjoint, suppose if possible that $z \in \mathbf{E}_i \cap
\mathbf{E}_j$. Then there are $x \in G_i, y \in G_j$ such that $z \in
B(x, r_x) \cap B(y, r_y)$. Supposing with no loss of generality that
$r_x \geq r_y$ we find that  
$$\dd(x,y) \leq \dd(x,z)+\dd(z,y) \leq 2r_x \leq \frac23 d \left(
  x, \partial \left(G_j \sm Z \right)\right),
$$
which is impossible. }
Let $G = \bigcup_{i \in \crly{I}} G_i$. Without loss of generality we can
assume  that $X= G \cup Z$. Define $F_i \df G_i \sm Z, F \df G \sm Z.$
Then the sets $F_i$ are closed and disjoint, and since the collection
$\{G_i\}$ is locally finite, the sets $F \sm F_i$ are closed as
well. Denote by $\mathbf{d}$ the metric on $X$ as well as the distance
from a point to a closed subset. Then it is easy to verify that the
sets 
$$
\mathbf{E}_i \df \{x \in G_i: \mathbf{d}(x, F_i) < \mathbf{d}(x, F \sm
F_i)\}
$$
satisfy the requirements. 
\end{proof}
We denote the nerve of a cover $\cA$ by 
  $\nerve (\mathcal{ A}),$ and consider it with the metric topology
  induced by barycentric coordinates.
 % \marg{we should find a reference for those who are not familiar with 
  %the terminology} 
  Given a partitition of unity subordinate 
  to a cover $\mathcal{A}$ of $X$, there is a standard construction of
  a map $X \to \nerve(\cA)$; such a
  map is called a {\em canonical map}. 

 \begin{lemma}
\Name{p1}
Let a locally connected metric space $Y$ be the union of two open subsets $\mathbf{D}$ and $\mathbf{E}$, and
let $\mathcal D$ and $\mathcal E$ be open covers of $\mathbf{D}$ and
$\mathbf{E}$ respectively, 
with bounded $\mesh$ and $\order$, and such that if $C\subset \mathbf{D}\cap\mathbf{E}$ is a connected subset 
contained in an element of $\cD$, then it is contained in an element of $\cE$.
Then, there is an open cover $\cY$ of $Y$ such that:
\begin{enumerate}
\item The cover $\mathcal{Y}$ refines
$\cD\cup \cE$.
\item $\mesh ({\mathcal Y}) \leq 
\max \left( \mesh(\cD), \mesh(\cE) \right ).$ 
\item $\order ({\mathcal Y}) \leq \max \left ( \order ({\mathcal D})+1, \order
\left({\mathcal E}\right) \right )$.
\end{enumerate}
\end{lemma}

\begin{proof} 
Let 
$\order ({\mathcal D})=n+1$, let $\mathbf{A} \df \nerve ({\mathcal D})$, and  let $\pi : \mathbf{D} \to \mathbf{A}$
be a canonical map. Take an open cover of ${\mathcal A}$ of $\mathbf{A}$ such that
$\order ({\mathcal A})\leq n+1$ and $\pi^{-1}({\mathcal A})$ refines ${\mathcal D}$.
Let $f : Y \to [0,1]$ be a continuous map such that
$f|_{Y \smallsetminus \mathbf{E}} \equiv 0$ and $f|_{Y\smallsetminus
  \mathbf{D}} \equiv 1$. Set
 $\mathbf{C} \df f^{-1} \left( \left[0,1 \right)\right )$
and 
$$g : \mathbf{C} \to \mathbf{B} \df \mathbf{A}\times \left [0,1 \right], \ \ \
g(c) \df (\pi(c), f(c)).$$
Since $\dim \mathbf{B} \leq n+1$ there is an open cover $\mathcal B$ of $\mathbf{B}$
such that 
$\order ({\mathcal B}) \leq n+2$,
the projection of  $\mathcal B$ to $\mathbf{A}$ refines ${\mathcal A}$
and the projection of $\mathcal B$ to 
$\left[0,1\right]$ is of $\mesh < 1/2$.
Let ${\mathcal C}$ denote the collection of connected components of
sets $\{g^{-1}(B): B \in \mathcal{B}\}$. By construction $\mathcal{C}$
refines $\mathcal D$. Moreover $\order(\mathcal{C}) \leq n+2$ and no
element of $\mathcal{C}$ meets both $f^{-1}(0)$ and $f^{-1} \left(
  \left[\frac{1}{2}, 1 \right]\right).$ Then for every $C \in
\mathcal{C}$ which meets  $f^{-1} \left(
  \left[\frac{1}{2}, 1 \right]\right)$ we have that $C \subset
\mathbf{D} \cap \mathbf{E} $ and there is an element $D \in
\mathcal{D}$ such that $C \subset D$ and hence there is $E \in \cE$
such that $C \subset E$. We  choose one such $E$ and say that $E$ {\em
  marks} $C$. 

We now modify elements of $\cE$, defining 
$$
\tilde{E} \df \left( E \cap f^{-1} \left( \left( \frac12,1 \right]
  \right) \right) \cup \bigcup_{E \text{ marks } C} C.
$$
Finally define $\mathcal Y$ as the collection of modified elements of 
$\cE$ and the elements of ${\mathcal C}$
which do not meet 
$f^{-1} \left( \left[\frac12, 1 \right]\right) $. 
It is easy to see that $\mathcal Y$ has the required properties. 
\end{proof}

\begin{lemma}\Name{lemma0036}
Let $Y$ be a locally connected metric space and let $\mathbf{D},\mathbf{E}_i, i\in
\crly{I}$ be open subsets which cover $Y$. Assume 
that the $\mathbf{E}_i$'s are disjoint, connected, and are uniformly
of $\on{asdim} \le \ell $.
Let $\cD$ be an open cover of $\mathbf{D}$ which is of bounded mesh
and $\on{ord}\cD\le \ell$. 
Then $Y$ has an open cover $\cY$ which refines the cover
$\cD\cup\set{\mathbf{E}_i:i\in\crly{I}}$, 
is of bounded
mesh and $\on{ord}\cY \le \ell+1$.
\end{lemma}
\begin{proof}
Using the assumption that $\mathbf{E}_i$ is uniformly of  $\on{asdim} \le \ell$ we find 
an open cover $\cE_i$ of $\mathbf{E}_i$ which is of uniformly bounded
mesh, such that $\on{ord}\cE_i\le \ell+1$ and
$\on{Leb} \cE_i> \on{mesh} \cD$. We assume that the sets in $\cE_i$
are subsets of $\mathbf{E}_i$ and $\Lb(\cE_i)$ is determined with
respect to the metric of $Y$ restriced to $\mathbf{E}_i$.  
Let
$$
\mathbf{E}\defi\bigcup_{i \in \crly{I}} \mathbf{E}_i \ \text{ and } \ 
\cE\defi\bigcup_{i \in \crly{I}}
\cE_i.
$$

Clearly it suffices to verify that the hypotheses of 
Lemma~\ref{p1} are satisfied. Indeed, 
by assumption the cover $\cD$ is of bounded $\mesh$ and order, and 
$\cE$ is of bounded mesh because of the
uniform bound on $\mesh(\cE_i)$. We  also have that
$\order \cE\le \ell+1$ because of the bounds $\order \cE_i\le\ell+1$ 
and the fact that the $\mathbf{E}_i$ are disjoint. %'s and the
For 
the last condition, % that we need to check, 
let a connected subset
$C\subset \mathbf{D}\cap\mathbf{E}$ which is contained in an element  
of $\cD$ be given. By the connectedness and disjointness of the
$\mathbf{E}_i$'s we conclude that there exists $i$ with  
$C\subset \mathbf{E}_i$. Because $\on{Leb} \cE_i>\mesh \cD$ we deduce
that  since $C$ is contained in an element of $\cD$ it must be  
contained in an element of $\cE_i$ and in turn, it must be contained in an element of  
%$\cE_i|_{\mathbf{E}_i}$ which is an element of 
$\cE$. % by definition.
%
%The outcome of an application of Lemma~\ref{p1} is an open cover $\cY$ of $Y$ for which 
%(i) $\cY$ refines $\cD\cup\cE$ and therefore refines
%$\cD\cup\set{\mathbf{E}_i:i\in\crly{I}}$, (ii) $\mesh \cY\le\max\set{
%  \mesh\cD, R} $ or in 
%other words, $\cY$ is of bounded mesh, and (iii) $\order\cY\le \ell +1$ as desired.
\end{proof}

 \begin{proof}[Proof of Proposition \ref{p2}.] 
Proceeding inductively
in reverse order, for $k=m, \ldots, 1$ 
 we will construct 
 a uniformly bounded open cover ${\mathcal A}^k $ of $[{\mathcal
   A}]^k$ 
 such that
 $\order ({\mathcal A}^k) \leq m+1-k$ and 
  ${\mathcal A}^k$ refines the restriction of $\mathcal A$ to
  $[{\mathcal A}]^k$. The construction is  
 obvious for $k =m$. Namely, our hypothesis and Definition \ref{def:
   asdim} with $n=m-k=0$ mean that 
$[\mathcal{A}]^m$ has a cover of bounded mesh and  order 1, that is,
we can just  set ${\mathcal A}^{m}$ to be
 the connected components of $[{\mathcal A}]^{m}$. 
Assume that the construction
 is completed for $k+1$ and proceed to $k$ as follows. First notice
 that for two distinct $k$-intersections $A$ and $A'$ of $\mathcal A$
 the complements $A\smallsetminus [{\mathcal A}]^{k+1}$
 and $A' \smallsetminus [{\mathcal A}]^{k+1}$ are disjoint.
By Lemma \ref{lem: for p2}, 
we can cover $[{\mathcal A}]^k\smallsetminus[{\mathcal A}]^{k+1}$ by 
a collection $\set{\mathbf{E}_i:i\in\crly{I}}$ of disjoint connected open sets such that every
$\mathbf{E}_i$ is contained in a $k$-intersection of $\mathcal A$. In particular, the collection
$\set{\mathbf{E}_i:i\in\crly{I}}$ is uniformly of asdim $\le m-k$. We can therefore apply Lemma~\ref{lemma0036}
with the choices $Y=\br{\cA}^k, \mathbf{D}=\br{\cA}^{k+1}, \cD=\cA^{k+1}$, the collection $\set{\mathbf{E}_i:i\in\crly{I}}$, and
$\ell=m-k$, and obtain an open cover $\cY$ of $\br{\cA}^k$ of order $\le m-k+1$ that refines $\cD\cup\set{\mathbf{E}_i:i\in\crly{I}}$ and 
in particular, refines $\cA|_Y$. This completes the inductive step.

\end{proof}

\begin{proof}[Proof of Proposition \ref{p3}.] 
For every $A \in \mathcal A$ choose
 a vertex $v^A_i$ of $\Delta_i$ so that $p_i(A)$ does not intersect 
 the face of $\Delta_i$ opposite to $v^A_i$. Let $Y \df \nerve(\cA)$
 and let $f : X \to X$ be the
 composition of 
a canonical map  $ X \to Y$ and a map 
$Y\to X$
which is  linear on each simplex of $Y$
and sends the vertex of $Y$ related to $A \in \mathcal A$
to the point $(v_1^A, v_2^A) \in X$. 
Take a point $x \in \partial \Delta_1 \times \Delta_2$. Then 
$p_1(x)$ belongs to a face $\Delta'_1$ of $\Delta_1$ and hence
for every $A \in \mathcal A$ containing $x$ we have that $v^A_1 \in \Delta'_1$.
Thus both $x$ and $f(x)$ belong to $\Delta'_1 \times \Delta_2$.
 Applying the same argument to
 $\Delta_1 \times \partial \Delta_2$ we get that
 the boundary $\partial X$ is invariant under $f$ and $f$ restricted
 to $\partial X$ is homotopic to the identity map of $\partial X$.
% Let
% $n \df \dim \Delta_1+\dim \Delta_2$. 
 If $\order ( {\mathcal A}) \leq \dim \Delta_1 + \dim \Delta_2$ then
  $\dim Y \leq \dim X -1$ and hence
 there is an interior point $a$ of $X$ not covered by $f(X)$. Take 
 a retraction $r : X \smallsetminus \{ a\} \to \partial X$. Then
 the identity map of $\partial X$ factors up to homotopy through 
 the contractible space $X$ which contradicts the non-triviality
 of the reduced homology of $\partial X$. 
\end{proof}
See \S\ref{section: karasev} for another proof of Proposition
\ref{p3}. 
 
\section{Another argument for Theorem \ref{thm:
    covering}}\Name{section: karasev}
In this section we will sketch another proof of Theorem \ref{thm:
  covering}. The proof, suggested by Roman Karasev,  proceeds by
reducing the theorem to its special case $s=0$. 
%i.e. to the result proved by McMullen \cite[Thm. 5.1]{McMullenMinkowski}. 

Let $\Delta_{CFK}$ denote the {\em Coxeter-Freudenthal-Kuhn
simplex}  
$$\{(x_1, \ldots,
x_n) \in \R^n: 0 \leq x_1 \leq \cdots \leq x_n \leq 1\},$$ 
 and let $\Gamma$ be the group generated by isometric reflections of
 $\R^n$ in the facets of $\Delta_{CFK}$. Then it is known
 \cite{Coxeter} that $\Gamma$ acts discretely on $\R^n$ with
 fundamental domain $\Delta_{CFK}$ ($\Gamma$ is the so-called {\em
   affine Coxeter group of type $\tilde{A}_n$}). Using this fact, we
 prove Theorem \ref{thm: covering} as follows. 

Recall that the case $s=0$ of the Theorem was proved by McMullen, see 
\cite[Thm. 5.1]{McMullenMinkowski}. According to this result, a cover
$\cV$  of $\R^{m}$ with $\Lb(\cV)>0$  has order at least
$m+1$, provided it satisfies the following analogue of
\equ{09302}:
\begin{itemize}
\item[{\em (ii)'}]
{\em 
There is $R$ so that for 
connected component $V$ of the intersection of $k \leq m$ distinct 
%connected components of 
elements of $\cV$ is $(R, m-k)$-almost affine.}
\end{itemize}
We remark that McMullen assumed that $\cV$ has positive inradius,
i.e. there is $r>0$ such that for any $x \in \R^m$, there is an
element of $\cV$ containing the ball of radius $r$ around $x$. However
as we remarked above, this hypothesis is not essential. 

Setting $m \df s+t$, starting with a cover of $M$ satisfying (i)
and (ii) we will form a cover of $\R^{m}$ satisfying (ii)'. 

Clearly there is no
 loss of generality in assuming that $\Delta = \Delta_{CFK}.$ 
Let $\varphi: \R^{m} \to \Delta \times \R^t$ be the map which sends
$(x,y),$ where $x \in \R^s, y \in \R^t$ to $(x', y)$ where $x'$ is the
representative of the orbit $\Gamma x$ in $\Delta$. Let $\cV$ be
the cover of $\R^{s+t}$ obtained by pulling back the cover $\cU$. For
each $j$, let $x_j$ be the vertex of $\Delta$ opposite $F_j$ and let
$\Gamma_j$ be the finite subgroup of $\Gamma$ fixing $x_j$. Then
$\Gamma_j \Delta$ is a polytope all of whose boundary faces are images
of $F_j$ under $\Gamma_j$. In light of assumption (i), this implies
that any connected component of any $V \in \cV$ is within a uniformly
bounded distance of $\{v\} \times \R^t$ for some $v \in V$. Therefore
(ii)' holds for $\cV$, for $k \leq s$, while for $k>s$, (ii)' for
$\cV$ is implied by (ii) for $\cU$. By McMullen's theorem, the order of $\cV$ is at least
$s+t+1$, and therefore the same holds for $\cU$. 

 A similar argument, also suggested by Roman Karasev, gives another proof of Proposition
 \ref{p3}. Namely suppose that for $i=1,2$, $\Delta_i \subset \R^{n_i}$ is realized
 concretely as the Coxeter-Freudenthal-Kuhn simplex of dimension
 $n_i$. Let $\Gamma = \Gamma_1 \times \Gamma_2$ where $\Gamma_i$ is
 the group generated by reflections in the facets of $\Delta_i$. Then
 a cover of $\Delta_1 \times \Delta_2$ gives rise to a cover of
 $\R^{n_1+n_2}$ by open sets of uniformly bounded diameter, and hence
 Lebesgue's theorem implies that there is a point which is covered
 $n_1+n_2+1$ times.

%\bibliographystyle{alpha}
%\bibliography{elonbib}
\def\cprime{$'$} \def\cprime{$'$} \def\cprime{$'$}

\end{document}